\documentclass[12 pt]{amsart}
\usepackage{amscd,amsfonts,amssymb,amsmath}
\usepackage{graphicx}
\usepackage{graphicx}
\newtheorem{theorem}{Theorem}[section]
\newtheorem{corollary}[theorem]{Corollary}
\newtheorem{lemma}[theorem]{Lemma}
\newtheorem{proposition}[theorem]{Proposition}
\theoremstyle{definition}

\newtheorem{question}[theorem]{Question}

\numberwithin{equation}{subsection}

\begin{document}
\title{On 3-strand singular pure braid group}

\author{Valeriy G. Bardakov}
\author{Tatyana A. Kozlovskaya}

\address{Sobolev Institute of Mathematics, 4 Acad. Koptyug avenue, 630090, Novosibirsk, Russia.}
\address{Novosibirsk State  University, 2 Pirogova Street, 630090, Novosibirsk, Russia.}
\address{Novosibirsk State Agrarian University, Dobrolyubova street, 160, Novosibirsk, 630039, Russia.}
\address{Tomsk State University, pr. Lenina, 36, Tomsk, 634050, Russia.}
\email{bardakov@math.nsc.ru}

\address{Regional Scientific and Educational Mathematical Center of Tomsk State University, pr. Lenina, 36, Tomsk, 634050, Russia.}
\email{t.kozlovskaya@math.tsu.ru}


\begin{abstract}
In the present paper we study  the singular pure braid group $SP_{n}$ for $n=2, 3$. We find generators, defining relations and  the algebraical structure of these groups. In particular, we prove that  $SP_{3}$ is a semi-direct product $SP_{3} = \widetilde{V}_3 \leftthreetimes \mathbb{Z}$, where $\widetilde{V}_3$  is an HNN-extension with base group $\mathbb{Z}^2 * \mathbb{Z}^2$ and cyclic associated subgroups. We prove that the center $Z(SP_3)$ of $SP_3$  is a direct factor in $SP_3$.
\end{abstract}

\maketitle

\section{Introduction}

E. Artin  introduced the braid group in 1925 and gave a presentation for this group in 1947  \cite{A}. The braid groups have applications in a wide variety of areas of mathematics, physics and biology. The notion of singular braids was introduced independently by John Baez \cite{Baez} and Joan Birman \cite{Bir}. It was shown that such braids form a monoid $SB_n$ which contains the Artin braid group $B_n$. The research of singular knots has been mostly motivated by the study of Vassiliev invariants. This led to increase of research activities in knot theory at the time. Singular braids and singular knots have been investigated by many scientifics \cite{C1}-\cite{JL}, \cite{Z}.

The singular braid group $SG_n$ was introduced by F.~Fenn, E.~Keyman and C.~Rourke \cite{FKR}. They proved that the singular braid  monoid  $SB_n$ is embedded into the group $SG_n$.   The word problem for the  singular braid monoid with three strings was solved by A.~Jarai \cite{J} (see also \cite{DG1}). R. Corran solved the word \cite{C1} and conjugacy problems for singular Artin monoid \cite{C2}. L. Paris proved Birman's conjecture that the desingularization map  $\eta $: $S{B_n} \to \mathbb{Z}\left[ {{B_n}} \right]$ is injective  \cite{Par} . This also gives a solution for the word problem in the  singular braid monoid.
 O.~Dasbach and B.~Gemein \cite{DG} introduced the singular pure braid group that is a generalization of the pure braid group $P_n$, found the set of generators and defining relations for this group and established that this group can be constructed using successive HNN-extensions. Investigation of homological properties of singular braids was started in \cite{V1}. We refer to \cite{V2,V3}  for more details on singular braid groups and also other generalized braid groups.

Monoid and group of pseudo braids was presented  in \cite{BJW}  and proved that it is isomorphic to monoid and group of singular braids, respectively.

The paper is organized as follows. In Section \ref{prelim} we give basic definitions of braid group  $B_n$, pure braid group $P_n$ and singular braid monoid $SB_n$, also we recall the Reidemeister - Schreier method. In Section \ref{pure} we find a presentation of a singular pure braid group $SP_2$ and $SP_3$, using the Reidemeister - Schreier method. Usually properties of braid groups with small number of strings are different from the general case (see for example  \cite{BMVW} and \cite{DG1}). In another generating set for $SP_3$ was found in \cite{DG}.
Since $SP_3$ is normal in $SG_3$, conjugations by generator of $SG_3$ induce automorphisms of $SP_3$. We find these automorphisms in Proposition \ref{p4.1}.
 Then we study the structure of $SP_{3}$ (see Theorem \ref{th}) and prove that $SP_{3}$ is a semi-direct product $SP_{3} = \widetilde{V}_3 \leftthreetimes \mathbb{Z}$, where $\widetilde{V}_3$  is an HNN-extension with base group $\mathbb{Z}^2 * \mathbb{Z}^2$ and cyclic associated subgroups. Also, we  establish that the center $Z(SG_3) = Z(SP_3)$ is a direct factor in $SP_3$ (see Theorem~\ref{t3} and Corollary~\ref{c3}).

\section{Basic definitions}\label{prelim}

\subsection{Artin braid groups}\label{Artin }

The braid group $B_n$, $n\geq 2$, on $n$ strings can be defined as
a group generated by $\sigma_1,\sigma_2,\ldots,\sigma_{n-1}$ with the defining relations
\begin{center}
$\sigma_i \, \sigma_{i+1} \, \sigma_i = \sigma_{i+1} \, \sigma_i \, \sigma_{i+1},~~~ i=1,2,\ldots,n-2, $
\end{center}
\begin{center}
$\sigma_i \, \sigma_j = \sigma_j \, \sigma_i,~~~|i-j|\geq 2. $
\end{center}

There exists a homomorphism of $B_n$ onto the symmetric group $S_n$ on
$n$ letters. This homomorphism  maps
 $\sigma_i$ to the transposition  $(i,i+1)$, $i=1,2,\ldots,n-1$.
The kernel of this homomorphism is called the
{\it pure braid group} and denoted by
$P_n$. The group $P_n$ is generated by  $a_{ij}$, $1\leq i < j\leq n$.
These
generators can be expressed by the generators of
 $B_n$ as follows
$$
a_{i,i+1}=\sigma_i^2,
$$
$$
a_{ij} = \sigma_{j-1} \, \sigma_{j-2} \ldots \sigma_{i+1} \, \sigma_i^2 \, \sigma_{i+1}^{-1} \ldots
\sigma_{j-2}^{-1} \, \sigma_{j-1}^{-1},~~~i+1< j \leq n.
$$

The subgroup $P_n$ is normal in $B_n$, and the quotient $B_n / P_n$ is the symmetric group $S_n$. The generators of $B_n$ act on the generator $a_{ij} \in P_n$ by the rules:
 \begin{center}
$\sigma_k^{-1} a_{ij} \sigma_k =  a_{ij},  ~\mbox{for}~k \not= i-1, i, j-1, j,$ \\
$\sigma_{i}^{-1} a_{i,i+1} \sigma_{i} =  a_{i,i+1}, $  \\
$ \sigma_{i-1}^{-1} a_{ij} \sigma_{i-1} =   a_{i-1,j},$   \\
$ \sigma_{i}^{-1} a_{ij} \sigma_{i} =  a_{i+1,j} [a_{i,i+1}^{-1}, a_{ij}^{-1}],  ~\mbox{for}~j \not= i+1, $\\
$ \sigma_{j-1}^{-1} a_{ij} \sigma_{j-1} =  a_{i,j-1},$   \\
 $\sigma_{j}^{-1} a_{ij} \sigma_{j} =  a_{ij} a_{i,j+1} a_{ij}^{-1},$
\end{center}
where $[a, b] = a^{-1} b^{-1} a b = a^{-1} a^b$.

Denote by
$$
U_{i} = \langle a_{1i}, a_{2i}, \ldots, a_{i-1,i} \rangle,~~~i = 2, \ldots, n,
$$
a subgroup of $P_n$.
It is known that $U_i$ is a free group of rank $i-1$. The pure braid group  $P_n$ is defined by the relations  (for $\varepsilon = \pm 1$):
 \begin{center}
$a_{ik}^{-\varepsilon} a_{kj}  a_{ik}^{\varepsilon} = (a_{ij} a_{kj})^{\varepsilon} a_{kj} (a_{ij} a_{kj})^{-\varepsilon}, $ \\
$ a_{km}^{-\varepsilon} a_{kj}  a_{km}^{\varepsilon} = (a_{kj} a_{mj})^{\varepsilon} a_{kj} (a_{kj} a_{mj})^{-\varepsilon},  ~\mbox{for}~m < j, $ \\
$ a_{im}^{-\varepsilon} a_{kj}  a_{im}^{\varepsilon} = [a_{ij}^{-\varepsilon}, a_{mj}^{-\varepsilon}]^{\varepsilon} a_{kj} [a_{ij}^{-\varepsilon}, a_{mj}^{-\varepsilon}]^{-\varepsilon},  ~\mbox{for}~i < k < m, $ \\
$a_{im}^{-\varepsilon} a_{kj} a_{im}^{\varepsilon} = a_{kj},  ~\mbox{for}~k < i < m < j ~\mbox{or}~  m < k. $
 \end{center}

The group $P_n$ is the semi--direct product of  the normal subgroup
$U_n$ and $P_{n-1}$. Similarly, $P_{n-1}$ is the semi--direct product of the free group
$U_{n-1}$  and $P_{n-2},$ and so on.
Therefore, $P_n$ is decomposable (see \cite{Mar}) into the following semi--direct product
$$
P_n=U_n\rtimes (U_{n-1}\rtimes (\ldots \rtimes
(U_3\rtimes U_2))\ldots),~~~U_i\simeq F_{i-1}, ~~~i=2,3,\ldots,n.
$$

\subsection{Singular braid groups}\label{Singular }

{\it The Baez--Birman monoid}
\cite{Baez, Bir} or {\it the singular braid monoid} $SB_n$ is generated
(as a monoid) by elements $\sigma_i,$ $\sigma_i^{-1}$, $\tau_i$, $i = 1, 2, \ldots, n-1$.
The elements $\sigma_i,$ $\sigma_i^{-1}$  generate the braid group
$B_n$. The generators $\tau_i$  satisfy the defining relations
\begin{center}
$\tau_i \, \tau_j = \tau_j \, \tau_i, ~~~|i - j| \geq 2, $
\end{center}
other relations are mixed:
\begin{center}
$\tau_{i} \, \sigma_{j} = \sigma_{j} \, \tau_{i}, ~~~|i - j| \geq 2, $
\end{center}
\begin{center}
$\tau_{i}  \, \sigma_{i} = \sigma_{i} \, \tau_{i},~~~ i=1,2,\ldots,n-1,  $
\end{center}
\begin{center}
$\sigma_{i} \, \sigma_{i+1} \, \tau_i = \tau_{i+1} \, \sigma_{i}  \, \sigma_{i+1},~~~ i=1,2,\ldots,n-2,$
 \end{center}
 \begin{center}
$\sigma_{i+1}  \, \sigma_{i} \, \tau_{i+1} = \tau_{i} \,
\sigma_{i+1} \, \sigma_{i}, ~~~ i=1,2,\ldots,n-2.$
\end{center}

In the work \cite{FKR} it was proved that the singular braid monoid $SB_n$ is embedded into
the group $SG_n$ which is called the  {\it
singular braid group} and has the same defining relations as $SB_n$.

$SB_n$ is generated by the unit, by the classical elementary braids $\sigma_i$ with their inverses and by the corresponding elementary singular braids $\tau_i$ (view Figure~\ref{figure}.):

\begin{figure}[h]
\centering{
\includegraphics[totalheight=4.cm]{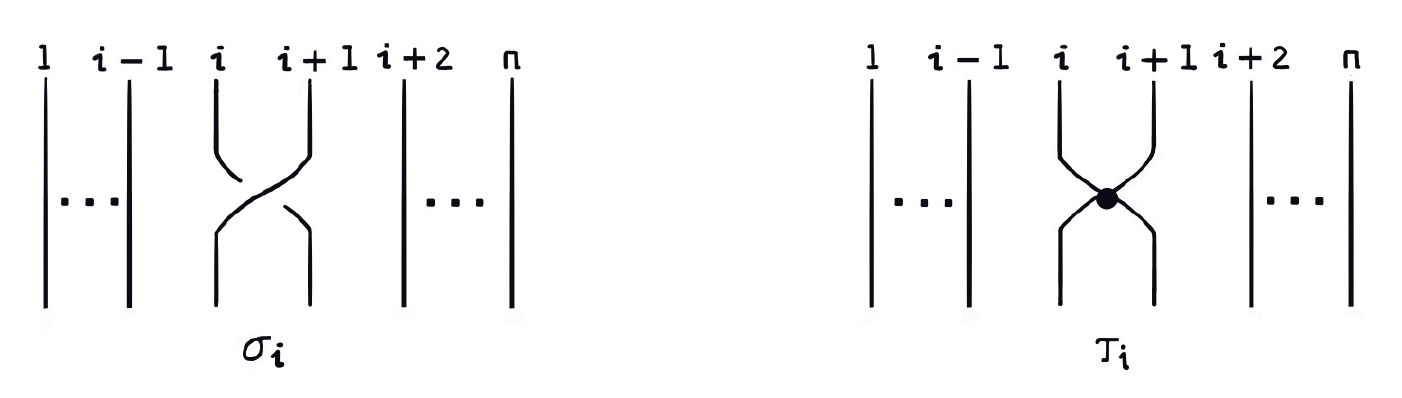}
\caption{The elementary braids $\sigma_i$ and $\tau_i$  .} \label{figure}
}
\end{figure}

\subsection{Singular pure braid group}\label{pure}

Define the map
$$
\pi : SG_n \longrightarrow S_n
$$
of $SG_n$ onto the symmetric group $S_n$ on $n$ symbols by actions the on generators
$$
\pi(\sigma_i) = \pi(\tau_i) = (i, i+1), ~~~ i = 1, 2, \ldots, n-1.
$$
The kernel $\mbox{ker}(\pi)$ of this map is called the
{\it singular pure braid group} and denoted by $SP_n$ (see \cite{DG}).
It is clear that $SP_n$ is a normal subgroup of index $n!$ of $SG_n$ and we have the short exact sequence
$$
1 \to SP_n \to SG_n \to S_n \to 1.
$$

To find a presentation of $SP_n$  we will use
the Reidemeister--Schreier method (see, for example, \cite[Ch. 2.2]{KMS}).

\subsection{ Reidemeister - Schreier method}\label{method}

Let $m_{kl} = \rho_{k-1}  \, \rho_{k-2} \ldots \rho_l$ for $l < k$ and $m_{kl} = 1$
in other cases. Then the set
$$
\Lambda_n = \left\{ \prod\limits_{k=2}^n m_{k,j_k} \vert 1 \leq j_k
\leq k \right\}
$$
is a Schreier set of coset representatives of $SP_n$ in $SG_n$.

Define the map $^- : SG_n \longrightarrow \Lambda_n$ which takes an element
$w \in SG_n$
into the representative $\overline{w}$ from $\Lambda_n$. In this case the element
$w \overline{w}^{-1}$ belongs to $SP_n$. By Theorem 2.7 from  \cite{KMS}
the group $SP_n$ is generated by
$$
S_{\lambda, a} = \lambda a \cdot (\overline{\lambda a})^{-1},
$$
where $\lambda$ runs over the set $\Lambda_n$ and $a$ runs over the set of generators of
$SG_n$.

To find  defining relations of $SP_n$ we define
a rewriting process $\tau $. It allows us to rewrite a word which is written in the generators
of $SG_n$ and presents an element in $SP_n$ as a word in the generators of $SP_n$.
Let us associate to the reduced word
$$
u = a_1^{\varepsilon_1} \, a_2^{\varepsilon_2} \ldots
a_{\nu}^{\varepsilon_{\nu}},~~~\varepsilon_l = \pm 1,~~~a_l \in
\{\sigma_1, \sigma_2, \ldots, \sigma_{n-1}, \tau_1, \tau_2, \ldots, \tau_{n-1}
\},
$$
the word
$$
\tau(u) = S_{k_1,a_1}^{\varepsilon_1} \,  S_{k_2,a_2}^{\varepsilon_2}
\ldots S_{k_{\nu},a_{\nu}}^{\varepsilon_{\nu}}
$$
in the generators of $SP_n$, where $k_j$ is a representative of the ($j-1$)th
initial segment
of the word $u$ if $\varepsilon_j = 1$ and $k_j$ is a representative of the $j$th
initial segment of
the word $u$ if
$\varepsilon_j = -1$.

By \cite[Theorem 2.9]{KMS}, the group $SP_n$ is defined by relations
$$
r_{\mu,\lambda} = \tau (\lambda  \, r_{\mu} \,  \lambda^{-1}),~~~\lambda \in
\Lambda_n,
$$
where $r_{\mu}$ is the defining relation of $SG_n$.

\section{Generators and defining relations for $SP_2$ and $SP_3$}

In this section we will use the Reidemeister--Schreier method, which we described in the previous section.

\subsection{Case $n=2$.} In this case
$$
SG_2 = \langle \sigma_1, \tau_1~|~\sigma_1 \tau_1 = \tau_1 \sigma_1 \rangle \cong \mathbb{Z} \times \mathbb{Z}.
$$
The set of coset representatives:
$$
\Lambda_2 = \{ 1, \sigma_1 \}.
$$
The group $SP_2$ is generated by elements
$$
S_{\lambda,a} = \lambda a \cdot (\overline{\lambda a})^{-1},~~~\lambda \in \Lambda_2,~~a \in \{ \sigma_1, \tau_1 \}.
$$
Hence,
$$
S_{1,\sigma_1} = \sigma_1 \cdot (\overline{\sigma_1})^{-1} = \sigma_1 \cdot \sigma_1^{-1} = 1,
$$
$$
S_{1,\tau_1} = \tau_1 \cdot (\overline{\tau_1})^{-1} = \tau_1 \cdot \sigma_1^{-1},
$$
$$
S_{\sigma_1,\sigma_1} = \sigma_1^2 \cdot \overline{\sigma_1^2}^{-1} = \sigma_1^2 \cdot 1 = \sigma_1^2,
$$
$$
S_{\sigma_1,\tau_1} = \sigma_1 \tau_1 \cdot (\overline{\sigma_1 \tau_1})^{-1} = \sigma_1 \tau_1.
$$
We see that $SP_2$ is generated by three elements:
$$
S_{1,\tau_1} =  \tau_1  \sigma_1^{-1},~~~S_{\sigma_1,\sigma_1} =  \sigma_1^2,~~~S_{\sigma_1,\tau_1} = \sigma_1 \tau_1.
$$
The element $a_{12} = \sigma_1^2$ is a generator of the pure braid group $P_2$.

To find the set of defining relations for $SP_2$ take the relation $r = \sigma_1 \tau_1 \sigma_1^{-1} \tau_1^{-1}$ and using the rewritable process we get
$$
r = S_{1,\sigma_1} S_{\sigma_1,\tau_1} S_{\sigma_1,\sigma_1}^{-1} S_{1,\tau_1}^{-1} = S_{\sigma_1,\tau_1} S_{\sigma_1,\sigma_1}^{-1} S_{1,\tau_1}^{-1} = 1.
$$
Using this relation we can remove the generator
$$
S_{1,\tau_1} = S_{\sigma_1,\tau_1} S_{\sigma_1,\sigma_1}^{-1}.
$$

Relation
$$
\sigma_1 r \sigma_1^{-1} = S_{1,\sigma_1} S_{\sigma_1,\sigma_1} S_{1,\tau_1} S_{1,\sigma_1}^{-1} S_{\sigma_1,\tau_1}^{-1} S_{1,\sigma_1}^{-1} =
 S_{\sigma_1,\sigma_1} S_{1,\tau_1} S_{\sigma_1,\tau_1}^{-1} = 1.
$$
Applying the previous relation we get
$$
 S_{\sigma_1,\sigma_1} S_{\sigma_1,\tau_1} = S_{\sigma_1,\tau_1} S_{\sigma_1,\sigma_1}.
$$
Put $a_{12} =  S_{\sigma_1,\sigma_1}$, $b_{12} = S_{\sigma_1,\tau_1}$. Then, using the fact that $SG_2$ is abelian we get

\begin{lemma} \label{l3.1}
1) $SP_2 = \langle a_{12}, b_{12}~|~a_{12} \, b_{12} = b_{12} \, a_{12} \rangle \cong \mathbb{Z} \times \mathbb{Z};$

2) $SP_2$ is normal in $SG_2$ and the  action of $SG_2$ on $SP_2$ is defined by the formulas
$$
a_{12}^{\sigma_1} = a_{12},~~~b_{12}^{\sigma_1} = b_{12},
$$
$$
a_{12}^{\tau_1} = a_{12},~~~b_{12}^{\tau_1} = b_{12}.
$$
\end{lemma}

\subsection{Case $n=3$.} In this case $SG_3$ is generated by elements
$$
\sigma_1, \sigma_2, \tau_1, \tau_2,
$$
and is defined by relations
$$
\sigma_1 \tau_1 = \tau_1 \sigma_1,~~~\sigma_1 \sigma_2 \sigma_1 = \sigma_2 \sigma_1 \sigma_2,~~~\sigma_2 \tau_2 = \tau_2 \sigma_2,~~~
\sigma_1 \sigma_2 \tau_1 = \tau_2 \sigma_1 \sigma_2,~~~\sigma_2 \sigma_1 \tau_2 = \tau_1 \sigma_2 \sigma_1.
$$
The set of coset representatives:
$$
\Lambda_3 = \{ 1, \sigma_1, \sigma_2, \sigma_1 \sigma_2, \sigma_2 \sigma_1, \sigma_1 \sigma_2 \sigma_1  \}.
$$
The group $SP_3$ is generated by elements
$$
S_{\lambda,a} = \lambda a \cdot (\overline{\lambda a})^{-1},~~~\lambda \in \Lambda_3,~~a \in \{ \sigma_1, \sigma_2, \tau_1, \tau_2 \}.
$$
We find these elements
\begin{align*}
& S_{1,\sigma_1} = \sigma_1 \cdot (\overline{\sigma_1})^{-1} = \sigma_1 \cdot \sigma_1^{-1} = 1,\\
& S_{1,\sigma_2} = \sigma_2 \cdot (\overline{\sigma_2})^{-1} = \sigma_2 \cdot \sigma_2^{-1} = 1,\\
& S_{1,\tau_1} = \tau_1 \cdot (\overline{\tau_1})^{-1} = \tau_1 \cdot \sigma_1^{-1},\\
& S_{1,\tau_2} = \tau_2 \cdot (\overline{\tau_2})^{-1} = \tau_2 \cdot \sigma_2^{-1},\\
\end{align*}
\begin{align*}
& S_{\sigma_1,\sigma_1} = \sigma_1^2 \cdot \overline{\sigma_1^2}^{-1} = \sigma_1^2 \cdot 1 = \sigma_1^2,\\
& S_{\sigma_1,\sigma_2} = \sigma_1 \sigma_2 \cdot (\overline{\sigma_1 \sigma_2})^{-1} = 1,\\
& S_{\sigma_1,\tau_1} = \sigma_1 \tau_1 \cdot (\overline{\sigma_1 \tau_1})^{-1} = \sigma_1 \tau_1,\\
& S_{\sigma_1,\tau_2} = \sigma_1 \tau_2 \cdot (\overline{\sigma_1 \tau_2})^{-1} = \sigma_1 \tau_2 \sigma_2^{-1} \sigma_1^{-1},\\
\end{align*}
\begin{align*}
& S_{\sigma_2,\sigma_1} = \sigma_2 \sigma_1 \cdot (\overline{\sigma_2 \sigma_1})^{-1} = 1,\\
& S_{\sigma_2,\sigma_2} = \sigma_2^2 \cdot \overline{\sigma_2^2}^{-1} = \sigma_2^2 \cdot 1 = \sigma_2^2,\\
& S_{\sigma_2,\tau_1} = \sigma_2 \tau_1 \cdot (\overline{\sigma_2 \sigma_1})^{-1} = \sigma_2 \tau_1 \sigma_1^{-1} \sigma_2^{-1},\\
& S_{\sigma_2,\tau_2} = \sigma_2 \tau_2,\\
\end{align*}
\begin{align*}
& S_{\sigma_1 \sigma_2,\sigma_1} = \sigma_1 \sigma_2 \sigma_1 \cdot (\sigma_1 \sigma_2 \sigma_1)^{-1} = 1,\\
& S_{\sigma_1 \sigma_2,\sigma_2} = \sigma_1 \sigma_2^2 \sigma_1^{-1},\\
& S_{\sigma_1 \sigma_2,\tau_1} = \sigma_1 \sigma_2 \tau_1 \sigma_1^{-1} \sigma_2^{-1} \sigma_1^{-1},\\
&S_{\sigma_1 \sigma_2,\tau_2} = \sigma_1 \sigma_2 \tau_2 \sigma_1^{-1},\\
\end{align*}
\begin{align*}
& S_{\sigma_2 \sigma_1,\sigma_1} = \sigma_2 \sigma_1^2 \sigma_2^{-1},\\
& S_{\sigma_2 \sigma_1,\sigma_2} = \sigma_2 \sigma_1 \sigma_2 \sigma_1^{-1} \sigma_2^{-1} \sigma_1^{-1},\\
& S_{\sigma_2 \sigma_1,\tau_1} = \sigma_2 \sigma_1 \tau_1 \sigma_2^{-1},\\
& S_{\sigma_2 \sigma_1,\tau_2} = \sigma_2 \sigma_1 \tau_2 \sigma_1^{-1} \sigma_2^{-1} \sigma_1^{-1},\\
\end{align*}
\begin{align*}
& S_{\sigma_1 \sigma_2 \sigma_1, \sigma_1} = \sigma_1 \sigma_2 \sigma_1^2 \sigma_2^{-1} \sigma_1^{-1},\\
& S_{\sigma_1 \sigma_2 \sigma_1, \sigma_2} = \sigma_1 \sigma_2 \sigma_1 \sigma_2 \sigma_1^{-1} \sigma_2^{-1},\\
& S_{\sigma_1 \sigma_2 \sigma_1, \tau_1} = \sigma_1 \sigma_2 \sigma_1 \tau_1 \sigma_2^{-1} \sigma_1^{-1},\\
& S_{\sigma_1 \sigma_2 \sigma_1, \tau_2} = \sigma_1 \sigma_2 \sigma_1 \tau_2 \sigma_1^{-1} \sigma_2^{-1}.\\
\end{align*}
Hence, $SP_3$ is generated by elements
$$
S_{1,\tau_1},~~S_{1,\tau_2},~~~S_{\sigma_1,\sigma_1},~~~S_{\sigma_1,\tau_1},~~~S_{\sigma_1,\tau_2},~~~S_{\sigma_2,\sigma_2},~~~S_{\sigma_2,\tau_1},~~~
S_{\sigma_2,\tau_2},~~~S_{\sigma_1 \sigma_2,\sigma_2},~~~S_{\sigma_1 \sigma_2,\tau_1},~~~S_{\sigma_1 \sigma_2,\tau_2},
$$
$$
S_{\sigma_2 \sigma_1,\sigma_1},~~~S_{\sigma_2 \sigma_1,\sigma_2},~~~S_{\sigma_2 \sigma_1,\tau_1},~~~S_{\sigma_2 \sigma_1,\tau_2},~~~
S_{\sigma_1 \sigma_2 \sigma_1, \sigma_1},~~~S_{\sigma_1 \sigma_2 \sigma_1, \sigma_2},~~~S_{\sigma_1 \sigma_2 \sigma_1, \tau_1},~~~S_{\sigma_1 \sigma_2 \sigma_1, \tau_2}.
$$
We see that elements
$$
S_{\sigma_1,\sigma_1},~~~S_{\sigma_2,\sigma_2},~~~S_{\sigma_1 \sigma_2,\sigma_2},~~~S_{\sigma_2 \sigma_1,\sigma_1},~~~S_{\sigma_2 \sigma_1,\sigma_2},~~~
S_{\sigma_1 \sigma_2 \sigma_1, \sigma_1},~~~S_{\sigma_1 \sigma_2 \sigma_1, \sigma_2},
$$
contain only generators $\sigma_1$ and $\sigma_2$. We will show that these elements generate the pure braid group $P_3$.

Next find the set of defining relations.

\textbf{1)}Take the relation $r_1 = \sigma_1 \tau_1 \sigma_1^{-1} \tau_1^{-1}$. We  considered relation $r_1 = 1$ and $\sigma_1 r_1 \sigma_1^{-1} = 1$ in the previous case. Conjugating $r_1$ by other coset representatives, we get
$$
r_{1,\sigma_2} = \sigma_2 r_1 \sigma_2^{-1} = S_{1,\sigma_2} S_{\sigma_2,\sigma_1} S_{\sigma_2 \sigma_1,\tau_1} S_{\sigma_2 \sigma_1,\sigma_1}^{-1} S_{\sigma_2, \tau_1}^{-1} S_{1,\sigma_2}^{-1} =  S_{\sigma_2 \sigma_1,\tau_1} S_{\sigma_2 \sigma_1,\sigma_1}^{-1} S_{\sigma_2, \tau_1}^{-1} = 1,
$$
From this relation
$$
S_{\sigma_2, \tau_1} = S_{\sigma_2 \sigma_1,\tau_1} S_{\sigma_2 \sigma_1, \sigma_1}^{-1}.
$$
Hence, we can remove this relation and the generator $S_{\sigma_2, \tau_1}$.

Relation
$$
r_{1,\sigma_1\sigma_2} =  \sigma_1 \sigma_2 r_1 \sigma_2^{-1} \sigma_1^{-1} = S_{1,\sigma_1} S_{\sigma_1,\sigma_2} S_{\sigma_1 \sigma_2, \sigma_1} S_{\sigma_1 \sigma_2 \sigma_1,\tau_1} S_{\sigma_1 \sigma_2 \sigma_1, \sigma_1}^{-1} S_{\sigma_1 \sigma_2, \tau_1}^{-1} S_{\sigma_1, \sigma_2}^{-1} S_{1,\sigma_1}^{-1} =
$$
$$
= S_{\sigma_1 \sigma_2 \sigma_1,\tau_1} S_{\sigma_1 \sigma_2 \sigma_1, \sigma_1}^{-1} S_{\sigma_1 \sigma_2, \tau_1}^{-1}  = 1.
$$
From this relation
$$
S_{\sigma_1 \sigma_2, \tau_1} = S_{\sigma_1 \sigma_2 \sigma_1,\tau_1} S_{\sigma_1 \sigma_2 \sigma_1, \sigma_1}^{-1}.
$$
Hence, we can remove this relation and the generator $S_{\sigma_1 \sigma_2, \tau_1}$.

From the relation

\[{r_{1,{\sigma _2}{\sigma _1}}} = {S_{{\sigma _2}{\sigma _1},{\sigma _1}}}{S_{{\sigma _2},{\tau _1}}}S_{{\sigma _2}{\sigma _1},{\tau _1}}^{ - 1}=1\]
follows that

\[{S_{{\sigma _2}{\sigma _1},{\tau _1}}} = {S_{{\sigma _2}{\sigma _1},{\sigma _1}}}{S_{{\sigma _2},{\tau _1}}}\]
Using the formula for ${S_{{\sigma _2},{\tau _1}}}$, we get relation

\[{S_{{\sigma _2}{\sigma _1},{\tau _1}}}{S_{{\sigma _2}{\sigma _1},{\sigma _1}}} = {S_{{\sigma _2}{\sigma _1},{\sigma _1}}}{S_{{\sigma _2}{\sigma _1},{\tau _1}}}\]

Relation
$$
r_{1,\sigma_1\sigma_2\sigma_1} = \sigma_1 \sigma_2 \sigma_1 r_1 \sigma_1^{-1} \sigma_2^{-1} \sigma_1^{-1} =
S_{1,\sigma_1} S_{\sigma_1,\sigma_2} S_{\sigma_1 \sigma_2, \sigma_1} S_{\sigma_1 \sigma_2 \sigma_1, \sigma_1}
  S_{\sigma_1 \sigma_2,\tau_1} \cdot
  $$
$$
\cdot  S_{\sigma_1 \sigma_2, \sigma_1}^{-1} S_{\sigma_1 \sigma_2 \sigma_1, \tau_1}^{-1} S_{\sigma_1 \sigma_2, \sigma_1}^{-1} S_{\sigma_1, \sigma_2}^{-1}  S_{1,\sigma_1}^{-1} =
$$
$$
=  S_{\sigma_1 \sigma_2 \sigma_1, \sigma_1}
  S_{\sigma_1 \sigma_2,\tau_1} S_{\sigma_1 \sigma_2 \sigma_1, \tau_1}^{-1} = 1.
$$
Since
$$
S_{\sigma_1 \sigma_2 \sigma_1, \sigma_1} = S_{\sigma_2, \sigma_2},~~~ S_{\sigma_1 \sigma_2,\tau_1} = S_{\sigma_1 \sigma_2 \sigma_1, \tau_1}
S_{\sigma_2, \sigma_2}^{-1},
$$
we get the relation
$$
S_{\sigma_2, \sigma_2} S_{\sigma_1 \sigma_2 \sigma_1, \tau_1}   S_{\sigma_2, \sigma_2}^{-1} = S_{\sigma_1 \sigma_2 \sigma_1, \tau_1}.
$$

\medskip

\begin{lemma} \label{l1}
From relation $r_1 = \sigma_1 \tau_1 \sigma_1^{-1} \tau_1^{-1}$ follows 6 relations, applying which we can remove generators:
$$
S_{1,\tau_1} = S_{\sigma_1,\tau_1} S_{\sigma_1,\sigma_1}^{-1},~~~S_{\sigma_2, \tau_1} = S_{\sigma_2 \sigma_1,\tau_1} S_{\sigma_2 \sigma_1, \sigma_1}^{-1},~~~S_{\sigma_1 \sigma_2, \tau_1} = S_{\sigma_1 \sigma_2 \sigma_1,\tau_1} S_{\sigma_1 \sigma_2 \sigma_1, \sigma_1}^{-1},
$$
and we get 3 relations:
$$
 S_{\sigma_1,\sigma_1} S_{\sigma_1,\tau_1} = S_{\sigma_1,\tau_1} S_{\sigma_1,\sigma_1},
$$
$$
{S_{{\sigma _2}{\sigma _1},{\tau _1}}}{S_{{\sigma _2}{\sigma _1},{\sigma _1}}} = {S_{{\sigma _2}{\sigma _1},{\sigma _1}}}{S_{{\sigma _2}{\sigma _1},{\tau _1}}},
$$
$$
S_{\sigma_2, \sigma_2} S_{\sigma_1 \sigma_2 \sigma_1, \tau_1}   S_{\sigma_2, \sigma_2}^{-1} = S_{\sigma_1 \sigma_2 \sigma_1, \tau_1}.
$$
\end{lemma}

\medskip

\textbf{2)} Take the relation $r_2 = \sigma_1 \sigma_2 \sigma_1 \sigma_2^{-1} \sigma_1^{-1} \sigma_2^{-1}$. Then
$$
r_2 = r_{2,1} = S_{1,\sigma_1} S_{\sigma_1,\sigma_2} S_{\sigma_1 \sigma_2, \sigma_1} S_{\sigma_2 \sigma_1, \sigma_2}^{-1}
S_{\sigma_2, \sigma_1}^{-1} S_{1, \sigma_2}^{-1} =  S_{\sigma_2 \sigma_1, \sigma_2}^{-1} = 1,
$$
i.e. $S_{\sigma_2 \sigma_1, \sigma_2} = 1$ and we can remove this generator.

Conjugating this relation by $\sigma_1^{-1}$, we get
$$
r_{2,\sigma_1} =  S_{1,\sigma_1} S_{\sigma_1,\sigma_1} S_{1,\sigma_2} S_{\sigma_2, \sigma_1} S_{\sigma_1 \sigma_2 \sigma_1, \sigma_2}^{-1}
S_{\sigma_1 \sigma_2, \sigma_1}^{-1} S_{\sigma_1, \sigma_2}^{-1} S_{1, \sigma_1}^{-1} =  S_{\sigma_1, \sigma_1} S_{\sigma_1 \sigma_2 \sigma_1, \sigma_2}^{-1} = 1,
$$
i.e. $S_{\sigma_1 \sigma_2 \sigma_1, \sigma_2} = S_{\sigma_1, \sigma_1}$ and we can remove $S_{\sigma_1 \sigma_2 \sigma_1, \sigma_2}$.

Conjugating $r_2$ by $\sigma_2^{-1}$, we get
$$
r_{2,\sigma_2} =  S_{1,\sigma_2} S_{\sigma_2,\sigma_1}  S_{\sigma_2 \sigma_1, \sigma_2} S_{\sigma_1 \sigma_2 \sigma_1, \sigma_1}
S_{\sigma_1, \sigma_2}^{-1} S_{1,\sigma_1}^{-1} S_{\sigma_2, \sigma_2}^{-1}  S_{1,\sigma_2}^{-1} =  S_{\sigma_2 \sigma_1, \sigma_2} S_{\sigma_1 \sigma_2 \sigma_1, \sigma_1}   S_{ \sigma_2, \sigma_2}^{-1} = 1.
$$
Since $S_{\sigma_2 \sigma_1, \sigma_2} = 1$, from this relation follows that
$S_{\sigma_1 \sigma_2 \sigma_1, \sigma_1} = S_{\sigma_2, \sigma_2}$ and we can remove $S_{\sigma_1 \sigma_2 \sigma_1, \sigma_1}$.

Conjugating $r_2$ by $(\sigma_1 \sigma_2)^{-1}$, we get

Relation
\[{r_{2,{\sigma _1}{\sigma _2}}} = {S_{{\sigma _1}{\sigma _2}{\sigma _1},{\sigma _2}}}{S_{{\sigma _2}{\sigma _1},{\sigma _1}}}S_{{\sigma _1},{\sigma _1}}^{ - 1}S_{{\sigma _1}{\sigma _2},{\sigma _2}}^{ - 1} = 1\]
gives relation
\[{S_{{\sigma _1}{\sigma _2}{\sigma _1},{\sigma _2}}}{S_{{\sigma _2}{\sigma _1},{\sigma _1}}} = {S_{{\sigma _1}{\sigma _2},{\sigma _2}}}{S_{{\sigma _1},{\sigma _1}}}\]
which is equivalent to
$$
{S_{{\sigma _1},{\sigma _1}}}{S_{{\sigma _2}{\sigma _1},{\sigma _1}}} = {S_{{\sigma _1}{\sigma _2},{\sigma _2}}}{S_{{\sigma _1},{\sigma _1}}}.
$$

From it we can remove the generator
$$
S_{{\sigma _1}{\sigma _2},{\sigma _2}} = S_{{\sigma _1},{\sigma _1}} S_{{\sigma _2}{\sigma _1},{\sigma _1}} S_{{\sigma _1},{\sigma _1}}^{-1}.
$$

Conjugating by $(\sigma _2 \sigma _1)^{-1}$ we get

\[{r_{2,{\sigma _2}{\sigma _1}}} = {S_{{\sigma _2}{\sigma _1},{\sigma _1}}}{S_{{\sigma _2},{\sigma _2}}}S_{{\sigma _1}{\sigma _2},{\sigma _2}}^{ - 1}S_{{\sigma _1}{\sigma _2}{\sigma _1},{\sigma _1}}^{ - 1} = 1\]
or
\[{S_{{\sigma _2}{\sigma _1},{\sigma _1}}}{S_{{\sigma _2},{\sigma _2}}} = {S_{{\sigma _1}{\sigma _2}{\sigma _1},{\sigma _1}}}{S_{{\sigma _1}{\sigma _2},{\sigma _2}}}.\]

From the previous relation

\[ S_{{\sigma _1},{\sigma _1}} {S_{{\sigma _2}{\sigma _1},{\sigma _1}}}{S_{{\sigma _1},{\sigma _1}}^{-1}} = {S_{{\sigma _2},{\sigma _2}}^{-1}} {S_{{\sigma _2}{\sigma _1}, {\sigma _1}}}{S_{{\sigma _2},{\sigma _2}}}.\]

We see that it is relation in $P_3$:
$$
a_{12} a_{13} a_{12}^{-1} = a_{23}^{-1} a_{13} a_{23}.
$$

Take the relation:

\[{r_{2,{\sigma _1}{\sigma _2}{\sigma _1}}} = {S_{{\sigma _1}{\sigma _2}{\sigma _1},{\sigma _1}}}{S_{{\sigma _1}{\sigma _2},{\sigma _2}}}{S_{{\sigma _1},{\sigma _1}}}S_{{\sigma _2},{\sigma _2}}^{ - 1}S_{{\sigma _2}{\sigma _1},{\sigma _1}}^{ - 1}S_{{\sigma _1}{\sigma _2}{\sigma _1},{\sigma _2}}^{ - 1} = 1,\]
that is equivalent to

\[{S_{{\sigma _1}{\sigma _2}{\sigma _1},{\sigma _1}}}{S_{{\sigma _1}{\sigma _2},{\sigma _2}}}{S_{{\sigma _1},{\sigma _1}}} = {S_{{\sigma _1}{\sigma _2}{\sigma _1},{\sigma _2}}}{S_{{\sigma _2}{\sigma _1},{\sigma _1}}}{S_{{\sigma _2},{\sigma _2}}}\]
or

\[{S_{{\sigma _2},{\sigma _2}}}{S_{{\sigma _1}{\sigma _2},{\sigma _2}}}{S_{{\sigma _1},{\sigma _1}}} = {S_{{\sigma _1},{\sigma _1}}}{S_{{\sigma _2}{\sigma _1},{\sigma _1}}}{S_{{\sigma _2},{\sigma _2}}}.\]
Hence we get
\[{S_{{\sigma _2},{\sigma _2}}}{S_{{\sigma _1},{\sigma _1}}}{S_{{\sigma _2}{\sigma _1},{\sigma _1}}} = {S_{{\sigma _1},{\sigma _1}}}{S_{{\sigma _2}{\sigma _1},{\sigma _1}}}{S_{{\sigma _2},{\sigma _2}}}.\]
This is relation in $P_3$:
$$
a_{23} a_{12} a_{13} = a_{12} a_{13} a_{23}.
$$
Taking into account the previous relation we get the following conjugation rule
$$
a_{12} a_{23} a_{12}^{-1} = a_{23}^{-1} a_{13}^{-1} a_{23} a_{13} a_{23}.
$$

\medskip

\begin{lemma} \label{l2}
From relation $r_2 = \sigma_1 \sigma_2 \sigma_1 \sigma_2^{-1} \sigma_1^{-1} \sigma_2^{-1}$ follows 6 relations, applying which we can remove 4 generators:
$$
S_{\sigma_2 \sigma_1, \sigma_2} = 1,~~~S_{\sigma_1 \sigma_2 \sigma_1, \sigma_2} = S_{\sigma_1, \sigma_1},~~~S_{\sigma_1 \sigma_2 \sigma_1, \sigma_1} = S_{\sigma_2, \sigma_2},~~~S_{{\sigma _1}{\sigma _2},{\sigma _2}} = S_{{\sigma _1},{\sigma _1}} S_{{\sigma _2}{\sigma _1},{\sigma _1}} S_{{\sigma _1},{\sigma _1}}^{-1},
$$
and we get 2 relations:
$$
S_{{\sigma _1},{\sigma _1}} {S_{{\sigma _2}{\sigma _1},{\sigma _1}}}{S_{{\sigma _1},{\sigma _1}}^{-1}} = {S_{{\sigma _2},{\sigma _2}}^{-1}} {S_{{\sigma _2}{\sigma _1}, {\sigma _1}}}{S_{{\sigma _2},{\sigma _2}}}
$$
$$
{S_{{\sigma _2},{\sigma _2}}}{S_{{\sigma _1},{\sigma _1}}}{S_{{\sigma _2}{\sigma _1},{\sigma _1}}} = {S_{{\sigma _1},{\sigma _1}}}{S_{{\sigma _2}{\sigma _1},{\sigma _1}}}{S_{{\sigma _2},{\sigma _2}}}.
$$
\end{lemma}

From the analysis of relations $r_{2,\lambda}$ follows

\begin{corollary}
The generators
$$
S_{{\sigma _1},{\sigma _1}} = a_{12},~~~S_{{\sigma _2}{\sigma _1}, {\sigma _1}} = a_{13},~~~S_{{\sigma _2},{\sigma _2}} = a_{23},
$$
satisfy relations
$$
a_{12} a_{13} a_{12}^{-1} = a_{23}^{-1} a_{13} a_{23},~~~a_{12} a_{23} a_{12}^{-1} = a_{23}^{-1} a_{13}^{-1} a_{23} a_{13} a_{23}.
$$
\end{corollary}

\medskip

\textbf{3)} Consider the relation $r_3 = \sigma_2 \tau_2 \sigma_2^{-1} \tau_2^{-1}$. From it

\[{r_{3,1}} = {S_{{\sigma _2},{\tau _2}}}S_{{\sigma _2},{\sigma _2}}^{ - 1}S_{1,{\tau _2}}^{ - 1}\]
and we can remove the generator

\[{S_{1,{\tau _2}}} = {S_{{\sigma _2},{\tau _2}}}S_{{\sigma _2},{\sigma _2}}^{ - 1}.\]

Conjugating by $(\sigma_1)^{-1}$
\[{r_{3,{\sigma _1}}} = {S_{{\sigma _1}{\sigma _2},{\tau _2}}}S_{{\sigma _1}{\sigma _2},{\sigma _2}}^{ - 1}S_{{\sigma _1},{\tau _2}}^{ - 1}\]
or
\[{S_{{\sigma _1},{\tau _2}}} = {S_{{\sigma _1}{\sigma _2},{\tau _2}}}S_{{\sigma _1}{\sigma _2},{\sigma _2}}^{ - 1}.\]
Take in attention the formula
$$
S_{{\sigma _1}{\sigma _2},{\sigma _2}} = S_{{\sigma _1},{\sigma _1}} S_{{\sigma _2}{\sigma _1},{\sigma _1}} S_{{\sigma _1},{\sigma _1}}^{-1},
$$
we can remove the generator
$$
{S_{{\sigma _1},{\tau _2}}} = {S_{{\sigma _1}{\sigma _2},{\tau _2}}} S_{{\sigma _1},{\sigma _1}} S_{{\sigma _2}{\sigma _1},{\sigma _1}}^{-1} S_{{\sigma _1},{\sigma _1}}^{-1}
$$

Conjugating by $(\sigma_2)^{-1}$
\[{r_{3,{\sigma _2}}} = {S_{{\sigma _2},{\sigma _2}}}{S_{1,{\tau _2}}}S_{{\sigma _2},{\tau _2}}^{ - 1} =
{S_{{\sigma _2},{\sigma _2}}} ( {S_{\sigma_2,{\tau _2}}}  S_{{\sigma _2},{\sigma _2}}^{ - 1} )  S_{{\sigma _2},{\tau _2}}^{ - 1} =1\]
and we have the relation
\[ {S_{{\sigma _2},{\sigma _2}}} {S_{{\sigma _2},{\tau _2}}} = {S_{{\sigma _2},{\tau _2}}} {S_{\sigma_2, {\sigma _2}}}.\]

Conjugating by $(\sigma _1 \sigma_2)^{-1}$
\[{r_{3,{\sigma _1}{\sigma _2}}} = {S_{{\sigma _1}{\sigma _2},{\sigma _2}}}{S_{{\sigma _1},{\tau _2}}}S_{{\sigma _1}{\sigma _2},{\tau _2}}^{ - 1} = 1\]
or
\[{S_{{\sigma _1}, {\sigma_1}}} {S_{{\sigma _2}{\sigma _1},{\sigma _1}}} {S_{{\sigma _1}, {\sigma_1}}^{-1}} \cdot {S_{{\sigma _1}{\sigma _2},{\tau _2}}} \cdot
{S_{{\sigma _1}, {\sigma_1}}} {S_{{\sigma _2}{\sigma _1},{\sigma _1}}^{-1}} {S_{{\sigma _1}, {\sigma_1}}^{-1}} = {S_{{\sigma _1}{\sigma _2},{\tau _2}}}.\]
This is relation in $SP_3$.

Conjugating by $(\sigma _2 \sigma_1)^{-1}$
\[{r_{3,{\sigma _2}{\sigma _1}}} = {S_{{\sigma _1}{\sigma _2}{\sigma _1},{\tau _2}}}S_{{\sigma _1}{\sigma _2}{\sigma _1},{\sigma _2}}^{ - 1}S_{{\sigma _2}{\sigma _1},{\tau _2}}^{ - 1} = 1.\]
We can remove the generator

\[{S_{{\sigma _2}{\sigma _1},{\tau _2}}} = {S_{{\sigma _1}{\sigma _2}{\sigma _1},{\tau _2}}}S_{{\sigma _1},{\sigma _1}}^{ - 1}.\]

Conjugating by $(\sigma_1 \sigma _2 \sigma_1)^{-1}$
\[{r_{3,{\sigma _1}{\sigma _2}{\sigma _1}}} = {S_{{\sigma _1}{\sigma _2}{\sigma _1},{\sigma _2}}}{S_{{\sigma _2}{\sigma _1},{\tau _2}}}S_{{\sigma _2}{\sigma _1},{\sigma _2}}^{ - 1}S_{{\sigma _1}{\sigma _2}{\sigma _1},{\tau _2}}^{ - 1} = 1\]
or

\[{r_{3,{\sigma _1}{\sigma _2}{\sigma _1}}} = {S_{{\sigma _1}{\sigma _2}{\sigma _1},{\sigma _2}}}{S_{{\sigma _2}{\sigma _1},{\tau _2}}}S_{{\sigma _1}{\sigma _2}{\sigma _1},{\tau _2}}^{ - 1} = 1,\]
i. e.

\[{S_{{\sigma _1}{\sigma _2}{\sigma _1},{\tau _2}}} = {S_{{\sigma _1},{\sigma _1}}}{S_{{\sigma _2}{\sigma _1},{\tau _2}}}.\]
Using the previous relation we get

\[{S_{{\sigma _1}{\sigma _2}{\sigma _1},{\tau _2}}}{S_{{\sigma _1},{\sigma _1}}} = {S_{{\sigma _1},{\sigma _1}}}{S_{{\sigma _1}{\sigma _2}{\sigma _1},{\tau _2}}}.\]

\medskip

\begin{lemma} \label{l3}
From relation $r_3 = \sigma_2 \tau_2 \sigma_2^{-1} \tau_2^{-1}$ follows 6 relations, applying which we can remove 3 generators:
$$
{S_{1,{\tau _2}}} = {S_{{\sigma _2},{\tau _2}}}S_{{\sigma _2},{\sigma _2}}^{ - 1},~~~{S_{{\sigma _1},{\tau _2}}} = {S_{{\sigma _1}{\sigma _2},{\tau _2}}} S_{{\sigma _1},{\sigma _1}} S_{{\sigma _2}{\sigma _1},{\sigma _1}}^{-1} S_{{\sigma _1},{\sigma _1}}^{-1},~~~
{S_{{\sigma _2}{\sigma _1},{\tau _2}}} = {S_{{\sigma _1}{\sigma _2}{\sigma _1},{\tau _2}}}S_{{\sigma _1},{\sigma _1}}^{ - 1},
$$
and we get 3 relations:
\[ {S_{{\sigma _2},{\sigma _2}}} {S_{{\sigma _2},{\tau _2}}} = {S_{{\sigma _2},{\tau _2}}} {S_{\sigma_2, {\sigma _2}}}.\]
$$
{S_{{\sigma _1}, {\sigma_1}}} {S_{{\sigma _2}{\sigma _1},{\sigma _1}}} {S_{{\sigma _1}, {\sigma_1}}^{-1}} \cdot {S_{{\sigma _1}{\sigma _2},{\tau _2}}} \cdot
{S_{{\sigma _1}, {\sigma_1}}} {S_{{\sigma _2}{\sigma _1},{\sigma _1}}^{-1}} {S_{{\sigma _1}, {\sigma_1}}^{-1}} = {S_{{\sigma _1}{\sigma _2},{\tau _2}}}
$$
$$
{S_{{\sigma _1}{\sigma _2}{\sigma _1},{\tau _2}}}{S_{{\sigma _1},{\sigma _1}}} = {S_{{\sigma _1},{\sigma _1}}}{S_{{\sigma _1}{\sigma _2}{\sigma _1},{\tau _2}}}.
$$
\end{lemma}

\textbf{4)} Take the relation $r_4 = \sigma_1 \sigma_2 \tau_1 \sigma_2^{-1} \sigma_1^{-1} \tau_2^{-1}$ and rewrite in the new generators:

\[{r_{4,1}} = {S_{{\sigma _1}{\sigma _2},{\tau _1}}}S_{1,{\tau _2}}^{ - 1} = 1\]
or
\[{S_{1,{\tau _2}}} = {S_{{\sigma _1}{\sigma _2},{\tau _1}}}.\]
Using the formulas for these generators from Lemmas \ref{l1} and \ref{l3} we get
\[{S_{{\sigma _2},{\tau _2}}} {S_{{\sigma _2},{\sigma _2}}^{-1}} = {S_{{\sigma _1}{\sigma _2}{\sigma _1},{\tau _1}}}
{S_{{\sigma _1}{\sigma _2}{\sigma _1},{\sigma _1}}^{-1}}.\]
Since
$$
{S_{{\sigma _1}{\sigma _2}{\sigma _1},{\sigma _1}}} = {S_{{\sigma _2}, {\sigma _2}}},
$$
we get the relation
$$
{S_{{\sigma _1}{\sigma _2}{\sigma _1},{\tau _1}}} = {S_{{\sigma _2},{\tau _2}}}.
$$
Applying this relation, we can remove ${S_{{\sigma _1}{\sigma _2}{\sigma _1},{\tau _1}}}$ from the generating set of $SP_3$.

Conjugating by $(\sigma_1)^{-1}$

\[{r_{4,{\sigma _1}}} = {S_{{\sigma _1},{\sigma _1}}}{S_{{\sigma _2},{\tau _1}}}S_{{\sigma _1}{\sigma _2}{\sigma _1},{\sigma _2}}^{ - 1} S_{{\sigma _1},{\tau _2}}^{ - 1} =
S_{{\sigma _1},{\sigma _1}} \cdot S_{{\sigma _2}{\sigma _1}, {\tau _1}}S_{{\sigma _2}{\sigma _1},{\sigma _1}}^{-1} \cdot S_{{\sigma _1},{\sigma _1}}^{ - 1} \cdot S_{{\sigma _1},{\sigma _1}}  S_{{\sigma _2}{\sigma _1},{\sigma _1}}  S_{{\sigma _1},{\sigma _1}}^{ - 1}
S_{{\sigma _1}{\sigma _2}, {\tau_2}}^{ - 1}= 1\]
and we can remove the generator
$$
S_{{\sigma _1}{\sigma _2}, {\tau _2}} = S_{{\sigma _1},{\sigma _1}} S_{{\sigma _2}{\sigma _1}, {\tau _1}} S_{{\sigma _1},{\sigma _1}}^{ - 1}.
$$

Next relation

\[{r_{4,{\sigma _2}}} = {S_{{\sigma _1}{\sigma _2}{\sigma _1},{\tau _1}}} S_{{\sigma _2},{\tau _2}}^{ - 1} = {S_{{\sigma _2},{\tau _2}}} S_{{\sigma _2},{\tau _2}}^{ - 1} = 1\]
is the trivial relation.

Also, it is easy to see that ${r_{4,{\sigma _1}{\sigma _2}}} = 1$ is the trivial relation.

Next relation
\[{r_{4,{\sigma _1}{\sigma _2}}} = {S_{{\sigma _2}{\sigma _1},{\sigma _1}}} S_{{\sigma _2}, {\sigma _2}} {S_{{1}, {\tau _1}}}
S_{{\sigma _1}{\sigma _2}, {\sigma _2}}^{ - 1} S_{{\sigma _1}{\sigma _2}{\sigma _1}, {\sigma _1}}^{ - 1} S_{{\sigma _2}{\sigma _1},{\tau _2}}^{ - 1} = 1.\]
Using formulas from the previous lemmas, we get
\[{S_{{\sigma _2}{\sigma _1}, {\sigma _1}}}  {S_{{\sigma _2},{\sigma _2}}} \cdot {S_{{\sigma _1},{\tau _1}}} S_{{\sigma _1},{\sigma _1}}^{ - 1} \cdot
S_{{\sigma _1},{\sigma _1}} S_{{\sigma _2}{\sigma _1}, {\sigma _1}}^{ - 1}  S_{{\sigma _1}, {\sigma _1}}^{ - 1}
S_{{\sigma _2}, {\sigma _2}}^{ - 1} \cdot
{S_{{\sigma _1}, {\sigma _1}}} {S_{{\sigma _1}{\sigma _2}{\sigma _1},{\tau _2}}^{-1}} = 1.\]

Applying this relation we can remove
$$
S_{{\sigma _1}{\sigma _2}{\sigma _1},{\tau _2}} = {S_{{\sigma _2}{\sigma _1}, {\sigma _1}}}  {S_{{\sigma _2},{\sigma _2}}}  {S_{{\sigma _1},{\tau _1}}} S_{{\sigma _2}{\sigma _1}, {\sigma _1}}^{ - 1}  S_{{\sigma _1}, {\sigma _1}}^{ - 1}
S_{{\sigma _2}, {\sigma _2}}^{ - 1}
{S_{{\sigma _1}, {\sigma _1}}}.
$$

The relation $r_{4,{\sigma _1}{\sigma _2}{\sigma _1}} = 1$ gives the trivial relation.

Hence, we have proven

\medskip

\begin{lemma} \label{l4}
From relation $r_4 = \sigma_1 \sigma_2 \tau_1 \sigma_2^{-1} \sigma_1^{-1} \tau_2^{-1}$ follows 3 non-trivial relations, applying which we can remove 3 generators:
$$
{S_{{\sigma _1}{\sigma _2}{\sigma _1},{\tau _1}}} = {S_{{\sigma _2},{\tau _2}}},~~~S_{{\sigma _1}{\sigma _2}, {\tau _2}} = S_{{\sigma _1},{\sigma _1}} S_{{\sigma _2}{\sigma _1}, {\tau _1}} S_{{\sigma _1},{\sigma _1}}^{ - 1},
$$
$$
S_{{\sigma _1}{\sigma _2}{\sigma _1},{\tau _2}} = {S_{{\sigma _2}{\sigma _1}, {\sigma _1}}}  {S_{{\sigma _2},{\sigma _2}}}  {S_{{\sigma _1},{\tau _1}}} S_{{\sigma _2}{\sigma _1}, {\sigma _1}}^{ - 1}  S_{{\sigma _1}, {\sigma _1}}^{ - 1}
S_{{\sigma _2}, {\sigma _2}}^{ - 1}
{S_{{\sigma _1}, {\sigma _1}}}.
$$
\end{lemma}

\medskip

\textbf{5)} Take the relation $r_5 = \sigma_2 \sigma_1 \tau_2 \sigma_1^{-1} \sigma_2^{-1} \tau_1^{-1}$ and rewrite in the new generators:

\[{r_{5,1}} = S_{{\sigma _2}{\sigma _1},{\tau _2}} S_{1,{\tau _1}}^{ - 1} =
S_{{\sigma _1}{\sigma _2}{\sigma _1}, {\tau _2}} S_{{\sigma _1},{\tau _1}}^{ - 1}= 1.\]
Using the formulas from the previous lemmas we get relation
$$
(a_{13} a_{23}) S_{{\sigma _1},{\tau _1}} = S_{{\sigma _1},{\tau _1}} (a_{13} a_{23}).
$$

Relation

\[{r_{5,{\sigma _1}}} = S_{{\sigma _1}{\sigma _2}{\sigma _1},{\tau _2}}
S_{{\sigma _1},{\tau _1}}^{ - 1} = 1\]
is the same as the previous relation.

Relation

\[{r_{5,{\sigma _2}}} = {S_{{\sigma _2},{\sigma _2}}} {S_{\sigma_1,{\tau _2}}} S_{{\sigma _1}{\sigma _2}{\sigma _1}, {\sigma _1}}^{ - 1} S_{{\sigma _2}, {\tau _1}}^{ - 1}  = {S_{{\sigma _2},{\sigma _2}}} {S_{{\sigma _1},{\tau _2}}} S_{{\sigma _2}, {\sigma _2}}^{ - 1} S_{{\sigma _2}, {\tau _1}}^{ - 1}.\]

Using the formulas from the previous lemmas we get relation
$$
S_{{\sigma _2} {\sigma _1}, {\tau _1}} a_{13}^{-1} = a_{23} a_{12} (S_{{\sigma _2} {\sigma _1}, {\tau _1}} a_{13}^{-1}) a_{12}^{-1} a_{23}^{-1}.
$$

From relation ${r_{5,{\sigma _1}{\sigma _2}}} =  1$ follows relation

$$
a_{12}^{-1} S_{{\sigma _2}, {\tau _2}} a_{12} = a_{13} S_{{\sigma _2}, {\tau _2}} a_{13}^{-1}.
$$

From $r_{5,{\sigma _2}{\sigma _1}} = 1$ follows relation
$$
a_{12} S_{{\sigma _2}{\sigma _1}, {\tau _1}} a_{12}^{-1} = a_{23}^{-1} S_{{\sigma _2}{\sigma _1}, {\tau _1}}  a_{23}.
$$

From ${r_{5,{\sigma _1}{\sigma _2}{\sigma _1}}} =  1$ follows relation
$$
a_{12}^{-1} S_{{\sigma _2}, {\tau _2}} a_{12} = a_{13} S_{{\sigma _2}, {\tau _2}}  a_{13}^{-1}
$$
and we see that it is the same relation which follows from ${r_{5,{\sigma _1}{\sigma _2}}} =  1$.

Hence we have proven

\medskip

\begin{lemma} \label{l5}
From relation $r_5 = \sigma_2 \sigma_1 \tau_2 \sigma_1^{-1} \sigma_2^{-1} \tau_1^{-1}$ follows  relations:
$$
(a_{13} a_{23}) S_{{\sigma _1},{\tau _1}} = S_{{\sigma _1},{\tau _1}} (a_{13} a_{23}),
$$
$$
S_{{\sigma _2} {\sigma _1}, {\tau _1}} a_{13}^{-1} = a_{23} a_{12} (S_{{\sigma _2} {\sigma _1}, {\tau _1}} a_{13}^{-1}) a_{12}^{-1} a_{23}^{-1},
$$
$$
a_{12}^{-1} S_{{\sigma _2}, {\tau _2}} a_{12} = a_{13} S_{{\sigma _2}, {\tau _2}}a_{13}^{-1},
$$
$$
a_{12} S_{{\sigma _2}{\sigma _1}, {\tau _1}} a_{12}^{-1} = a_{23}^{-1} S_{{\sigma _2}{\sigma _1}, {\tau _1}}  a_{23}.
$$
\end{lemma}

\medskip

Let us introduce the following notations
$$
a_{12} = S_{{\sigma _1},{\sigma _1}} = \sigma_1^2,~~~a_{13} = S_{{\sigma _2}{\sigma _1}, {\sigma _1}}  = \sigma_2 \sigma_1^2 \sigma_2^{-1},~~~a_{23} = S_{{\sigma _2},{\sigma _2}} = \sigma_2^2,
$$
$$
b_{12} = S_{{\sigma _1}, {\tau _1}} = \sigma _1 \tau _1,~~~{b_{13}} = {S_{{\sigma _2}{\sigma _1},{\tau _1}}} = {\sigma _2}{\sigma _1}{\tau _1}\sigma _2^{ - 1},~~~b_{23} = S_{{\sigma _2}, {\tau _2}} = \sigma _2 \tau _2.
$$

Then we can express other generators of $SP_3$.
\begin{lemma} The following equalities hold
\begin{align*}
& {S_{1,{\tau _1}}} = {\tau _1} \cdot \sigma _1^{ - 1} = {b_{12}}a_{12}^{ - 1},\\
& {S_{1,{\tau _2}}} = {\tau _2} \cdot \sigma _2^{ - 1} = {b_{23}}a_{23}^{ - 1},\\
& {S_{{\sigma _1},{\sigma _1}}} = \sigma _1^2 = {a_{12}},\\
&{S_{{\sigma _1},{\tau _1}}} = {\sigma _1}{\tau _1} = {b_{12}},\\
&{S_{{\sigma _1},{\tau _2}}} = {\sigma _1}{\tau _2}\sigma _2^{ - 1}\sigma _1^{ - 1} = {a_{23}^{-1}}{b_{13}}a_{13}^{ - 1}a_{23}, \\
&{S_{{\sigma _2},{\sigma _2}}} = \sigma _2^2 = {a_{23}}, \\
& {S_{{\sigma _2},{\tau _1}}} = {\sigma _2}{\tau _1}\sigma _1^{ - 1}\sigma _2^{ - 1} = {b_{13}}a_{13}^{ - 1},\\
& {S_{{\sigma _2},{\tau _2}}} = {\sigma _2}{\tau _2} = {b_{23}},\\
& {S_{{\sigma _1}{\sigma _2},{\sigma _2}}} = {\sigma _1}\sigma _2^2\sigma _1^{ - 1} = {a_{23}^{ - 1}}{a_{13}}a_{23},\\
& {S_{{\sigma _1}{\sigma _2},{\tau _1}}} = {\sigma _1}{\sigma _2}{\tau _1}\sigma _1^{ - 1}\sigma _2^{ - 1}\sigma _1^{ - 1} = {b_{23}}a_{23}^{ - 1},\\
&{S_{{\sigma _1}{\sigma _2},{\tau _2}}} = {\sigma _1}{\sigma _2}{\tau _2}\sigma _1^{ - 1} = {a_{23}^{ - 1}}{b_{13}}a_{23},\\
& {S_{{\sigma _2}{\sigma _1},{\sigma _1}}} = {\sigma _2}\sigma _1^2\sigma _2^{ - 1} = {a_{13}},\\
& {S_{{\sigma _2}{\sigma _1},{\tau _1}}} = {\sigma _2}{\sigma _1}{\tau _1}\sigma _2^{ - 1} = {b_{13}},\\
& {S_{{\sigma _2}{\sigma _1},{\tau _2}}} = {\sigma _2}{\sigma _1}{\tau _2}\sigma _1^{ - 1}\sigma _2^{ - 1}\sigma _1^{ - 1} = {b_{12}}a_{12}^{ - 1},\\
&{S_{{\sigma _1}{\sigma _2}{\sigma _1},{\sigma _1}}} = {\sigma _1}{\sigma _2}\sigma _1^2\sigma _2^{ - 1}\sigma _1^{ - 1} = {a_{23}},\\
& {S_{{\sigma _1}{\sigma _2}{\sigma _1},{\sigma _2}}} = {\sigma _1}{\sigma _2}{\sigma _1}{\sigma _2}\sigma _1^{ - 1}\sigma _2^{ - 1} = {a_{12}},\\
&{S_{{\sigma _1}{\sigma _2}{\sigma _1},{\tau _1}}} = {\sigma _1}{\sigma _2}{\sigma _1}{\tau _1}\sigma _2^{ - 1}\sigma _1^{ - 1} = {b_{23}},\\
& {S_{{\sigma _1}{\sigma _2}{\sigma _1},{\tau _2}}} = {\sigma _1}{\sigma _2}{\sigma _1}{\tau _2}\sigma _1^{ - 1}\sigma _2^{ - 1} = {b_{12}}.\\
\end{align*}
\end{lemma}

The group $SP_3$ has the following presentation.

\begin{theorem} \label{t1}
The singular pure braid group $SP_3$ is generated by elements
$$
a_{12},~~a_{13},~~a_{23},~~b_{12},~~b_{13},~~b_{23},
$$
and is defined by relations:
$$
a_{12} a_{13} a_{12}^{-1} = a_{23}^{-1} a_{13} a_{23},~~~a_{12} a_{23} a_{12}^{-1} = a_{23}^{-1} a_{13}^{-1} a_{23} a_{13} a_{23},
$$
-- what are relations in $P_3$;
$$
a_{12} b_{12} = b_{12} a_{12}
$$
-- what is the relation in $SP_2$;

$$
a_{13} b_{13} = b_{13} a_{13},
$$
$$
a_{23} b_{23} = b_{23} a_{23},
$$
$$
b_{12} (a_{13} a_{23}) b_{12}^{-1} = a_{13} a_{23},
$$
$$
a_{12} b_{13} a_{12}^{-1} =  a_{23}^{-1} b_{13} a_{23},
$$
$$
a_{12} b_{23} a_{12}^{-1} = a_{23}^{-1} a_{13}^{-1} b_{23} a_{13} a_{23}.
$$
\end{theorem}

\begin{corollary} \label{c2}
From these relations follow the conjugation rules
$$
a_{12}^{-1} a_{13} a_{12} = a_{13} a_{23} a_{13} a_{23}^{-1} a_{13}^{-1},~~~a_{12}^{-1} a_{23} a_{12}=  a_{13} a_{23} a_{13}^{-1},
$$
$$
a_{12}^{-1} b_{13} a_{12} = a_{13} a_{23} b_{13} a_{23}^{-1} a_{13}^{-1},~~~a_{12}^{-1} b_{23} a_{12}=  a_{13} b_{23} a_{13}^{-1}.
$$
\end{corollary}

\section{Some properties of $SP_3$}

\medskip

\subsection{Conjugations by elements of $SG_3$} Since $SP_3$ is normal in $SG_3$ conjugations by elements from $SG_3$ give automorphisms of $SP_3$. The following proposition gives conjugation formulas of generators $a_{12}$, $a_{13}$, $a_{23}$, $b_{12}$, $b_{13}$, $b_{23}$ by generators of $SG_3$.

\begin{proposition} \label{p4.1}
Generators of $SG_3$ act on the generators of $SP_3$ by the rules:

-- action of $\sigma_1$:

$$
a_{12}^{\sigma_1} = a_{12},~~~a_{13}^{\sigma_1} = a_{13} a_{23} a_{13}^{-1},~~~ a_{23}^{\sigma_1} = a_{13};
$$
$$
b_{12}^{\sigma_1} = b_{12},~~~b_{13}^{{\sigma _1}} = {a_{13}}{b_{23}}a_{13}^{ - 1},~~~ b_{23}^{\sigma _1} = b_{13};
$$

-- action of $\sigma_2$:

$$
a_{12}^{\sigma_2} = a_{23}^{-1} a_{13} a_{23},~~~a_{13}^{\sigma_2} = a_{12},~~~ a_{23}^{\sigma_2} = a_{23};
$$
$$
b_{12}^{{\sigma _2}} = a_{23}^{ - 1}{b_{13}}{a_{23}},~~~b_{13}^{{\sigma _2}} = {b_{12}},~~~ b_{23}^{\sigma_2} = b_{23};
$$

-- action of $\tau_1$:

$$
a_{12}^{\tau_1} = a_{12},~~~a_{13}^{\tau_1} = b_{12}^{-1} a_{23} b_{12},~~~a_{23}^{\tau_1} = b_{12}^{-1} a_{23}^{-1} a_{13} a_{23} b_{12},
$$
$$
b_{12}^{\tau_1} = b_{12},~~~b_{13}^{{\tau _1}} = b_{12}^{ - 1}{b_{23}}{b_{12}},~~~b_{23}^{{\tau _1}} = b_{12}^{ - 1}{a_{12}}{b_{13}}a_{12}^{ - 1}{b_{12}},
$$

-- action of $\tau_2$:

$$
a_{12}^{\tau _2} = b_{23}^{ - 1}{a_{13}}{b_{23}},~~~
a_{13}^{\tau _2} = b_{23}^{ - 1}{a_{23}}{a_{12}}a_{23}^{ - 1}{b_{23}},~~~a_{23}^{\tau _2} = a_{23},
$$

$$
b_{12}^{\tau _2} = b_{23}^{ - 1}{b_{13}}{b_{23}},~~~b_{13}^{\tau _2} =  b_{23}^{-1} a_{23} b_{12} a_{23}^{ - 1}{b_{23}},~~~
b_{23}^{\tau_2} = b_{23}.
$$

\end{proposition}

\begin{proof}
The formulas of conjugations of $a_{12}$, $a_{13}$, $a_{23}$ by $\sigma_1$ and $\sigma_2$ follow from conjugation rules in $B_3$.

Let us prove the other formulas:

$$
\tau_1^{-1} a_{13} \tau_1 = \tau_1^{-1} \sigma_2 \sigma_1 \sigma_1 \sigma_2^{-1} \tau_1 = S_{\sigma_1, \tau_1}^{-1}
S_{\sigma_1 \sigma_2 \sigma_1, \sigma_1}S_{\sigma_1, \tau_1} = b_{12}^{-1} a_{23} b_{12},
$$
$$
\tau_1^{-1} a_{23} \tau_1 = \tau_1^{-1} \sigma_2 \sigma_2  \tau_1 = S_{\sigma_1, \tau_1}^{-1}
S_{\sigma_1 \sigma_2, \sigma_2} S_{\sigma_1, \tau_1} = b_{12}^{-1} a_{12} a_{13} a_{12}^{-1} b_{12} = b_{12}^{-1} a_{23}^{-1} a_{13} a_{23} b_{12},
$$

\[\begin{gathered}
  \tau _2^{ - 1}{a_{13}}{\tau _2} = \tau _2^{ - 1}{\sigma _2}{\sigma _1}{\sigma _1}\sigma _2^{ - 1}{\tau _2} = S_{{\sigma _2},{\tau _2}}^{ - 1}{S_{{\sigma _1}{\sigma _2}{\sigma _1},{\sigma _1}}}{S_{{\sigma _1}{\sigma _2}{\sigma _1},{\sigma _2}}}S_{{\sigma _1}{\sigma _2}{\sigma _1},{\sigma _1}}^{ - 1}{S_{{\sigma _2},{\tau _2}}} =  \hfill \\
  \,\,\,\,\,\,\,\,\,\,\,\,\,\,\,\,\,\,\,\,\, = S_{{\sigma _2},{\tau _2}}^{ - 1}{S_{{\sigma _2},{\sigma _2}}}{S_{{\sigma _1},{\sigma _1}}}S_{{\sigma _2},{\sigma _2}}^{ - 1}{S_{{\sigma _2},{\tau _2}}} = b_{23}^{ - 1}{a_{23}}{a_{12}}a_{23}^{ - 1}{b_{23}}, \hfill \\
\end{gathered} \]

\[\tau _2^{ - 1}{b_{12}}{\tau _2} = \tau _2^{ - 1}{\sigma _1}{\tau _1}{\tau _2} = S_{{\sigma _2},{\tau _2}}^{ - 1}{S_{{\sigma _2}{\sigma _1},{\tau _1}}}{S_{{\sigma _2},{\tau _2}}} = b_{23}^{ - 1}{b_{13}}{b_{23}},\]
\[\begin{gathered}
  \tau _2^{ - 1}{b_{13}}{\tau _2} = \tau _2^{ - 1}{\sigma _2}{\sigma _1}{\tau _1}\sigma _2^{ - 1}{\tau _2} = S_{{\sigma _2},{\tau _2}}^{ - 1}{S_{{\sigma _1}{\sigma _2}{\sigma _1},{\sigma _1}}}{S_{{\sigma _1},{\tau _1}}}S_{{\sigma _1}{\sigma _2}{\sigma _1},{\sigma _1}}^{ - 1}{S_{{\sigma _2},{\tau _2}}} =  \hfill \\
  \,\,\,\,\,\,\,\,\,\,\,\,\,\, = S_{{\sigma _2},{\tau _2}}^{ - 1}{S_{{\sigma _2},{\sigma _2}}}{S_{{\sigma _1},{\tau _1}}}S_{{\sigma _2},{\sigma _2}}^{ - 1}{S_{{\sigma _2},{\tau _2}}} = b_{23}^{ - 1}{a_{23}}{b_{12}}a_{23}^{ - 1}{b_{23}}, \hfill \\
\end{gathered} \]

\end{proof}

\subsection{Decomposition  of $SP_3$} Define the following subgroups of $SP_3$:
$$
V_2 = SP_2 = \langle a_{12}, b_{12} \rangle,~~~V_3 = \langle a_{13}, a_{23}, b_{13}, b_{23} \rangle,~~~\widetilde{V}_3 = \langle V_3, b_{12} \rangle.
$$

As was proved in Lemma \ref{l3.1}
$$
 V_2 = SP_2 =  \langle a_{12}, b_{12}~||~ a_{12} b_{12} = b_{12} a_{12} \rangle \cong \mathbb{Z} \times \mathbb{Z}.
$$

There is  a homomorphism $SP_3 \to V_2$ that sends the generators $a_{13}, a_{23}, b_{13}, b_{23}$ to 1 and keeps the generators $a_{12}, b_{12}$. It is easily to see that the image of this homomorphism is $SP_2$ and the kernel is the normal closure of $V_3$ in  $SP_3$.

The following theorem describes a structure of $SP_3$.

\begin{theorem}  \label{th}
1) $SP_3 = \widetilde{V}_3 \leftthreetimes \mathbb{Z}$, where $\mathbb{Z} = \langle a_{12} \rangle$;

2) $\widetilde{V}_3$ has a presentation
$$
\widetilde{V}_3 =  \langle a_{13}, a_{23}, b_{13}, b_{23}, b_{12}~||~ [a_{13}, b_{13}] = [a_{23}, b_{23}] = [a_{13} a_{23}, b_{12}]= 1  \rangle
$$
and is an HNN-extension with base group $V_3$, stable letter $b_{12}$ and associated  subgroups $A \cong B = \langle a_{13} a_{23} \rangle$ and identity isomorphism $A \to B$:
$$
\widetilde{V}_3 = \langle V_3, b_{12}~|~rel(V_3),~~~b_{12}^{-1} (a_{13} a_{23}) b_{12} = a_{13} a_{23} \rangle,
$$
where $rel(V_3)$ is the set of relations in $V_3$.

3) $V_3 = \langle a_{13},  a_{23}, b_{13},  b_{23} ~||~[a_{13}, b_{13}] = [a_{23}, b_{23}] = 1 \rangle \cong \mathbb{Z}^2 * \mathbb{Z}^2$.
\end{theorem}

\begin{proof}
1) From the defining relations of $SP_3$ follows that $\widetilde{V}_3$ is normal in $SP_3$ and $\langle a_{12} \rangle$ acts on this subgroup by automorphisms.
Since $SP_3 = \widetilde{V}_3 \cdot \langle a_{12} \rangle$ and $\widetilde{V}_3 \cap \langle a_{12} \rangle = 1$, we have the desired decomposition.

2) To find a presentation of $\widetilde{V}_3$ define an endomorphism $\varphi : SP_3 \to \langle a_{12} \rangle$ that sends the generators $b_{12}, a_{13}, a_{23}, b_{13}, b_{23}$ to 1 and sends the generator $a_{12}$ to itself. To find $\mathrm{Ker}\varphi$ use the Reidemeister-Shraier method. The kernel
is the normal closure of $\widetilde{V}_3$ in $SP_3$, but since $\widetilde{V}_3$ is normal in $SP_3$, then $\mathrm{Ker}\varphi = \widetilde{V}_3$. To find defining relations of $\widetilde{V}_3$ take the set $\Lambda = \{ a_{12}^k ~|~k \in \mathbb{Z} \}$ as coset representatives of $\widetilde{V}_3$ in $SP_3$. Then $\widetilde{V}_3$ is defined by relations $\lambda r \lambda^{-1} = 1$, where $r=1$ is a relation in $SP_3$ and $\lambda \in \Lambda$. From Theorem \ref{t1} and Corollary \ref{c2} follows that all these relations follow from relations
$$
[a_{13}, b_{13}] = [a_{23}, b_{23}] = [a_{13} a_{23}, b_{12}]= 1.
$$
The decomposition in HNN-extension follows from the definition.

3) From the properties of HNN-extension follows that the base group is embedded in the HNN-extension.

\end{proof}

\subsection{The center of  $SG_n$ }\label{center}

It is well-known that the center $Z(B_n) = Z(P_n)$ is infinite cyclic group that is generated by
$$
\delta_n = (\sigma_1 \sigma_2 \ldots \sigma_{n-1})^n = a_{12} (a_{13} a_{23}) \ldots (a_{1n} a_{2n} \ldots a_{n-1,n}).
$$

It was shown that  $Z(B_n) \cong Z(SG_n)$ (see \cite{FRZ,V4}). On the other side, M.~V.~Neshchadim \cite{N1,N2} proved that $Z(P_n)$ is a direct factor in $P_n$.

\begin{question}
Is it true that $Z(SG_n)$ is a direct factor in $SP_n$?
\end{question}

We will prove that it is true for $n = 3$. To do it denote $\delta = \delta_3 = a_{12} a_{13} a_{23}$. If we remove the generator $a_{12} = \delta a_{23}^{-1} a_{13}^{-1}$, then we get the following presentation of $SP_3$.

\begin{theorem} \label{t3}
The singular pure braid group $SP_3$ is generated by elements
$$
\delta,~~a_{13},~~a_{23},~~b_{12},~~b_{13},~~b_{23},
$$
and is defined by relations:
$$
\delta b_{12} = b_{12} \delta,~~~\delta a_{13} = a_{13} \delta,~~~\delta a_{23} = a_{23} \delta,~~~\delta b_{13} = b_{13} \delta,~~~
\delta b_{23} = b_{23} \delta,
$$
-- what are relations of commutativity with $\delta$;

$$
a_{13} b_{13} = b_{13} a_{13},~~~
a_{23} b_{23} = b_{23} a_{23},~~~
b_{12} (a_{13} a_{23}) b_{12}^{-1} = a_{13} a_{23},
$$
-- what are relations in  $\widetilde{V}_3$.
\end{theorem}

From this representation we get

\begin{corollary} \label{c3}
$SP_3$ is the direct product
$$
SP_3 = Z \times \widetilde{V}_3,
$$
where $Z = \langle \delta \rangle$ is the center of $SP_3$ and
$
\widetilde{V}_3 = \langle a_{13}, a_{23}, b_{12}, b_{13}, b_{23}  \rangle.
$
\end{corollary}

\emph{Acknowledgements. }The first and second named authors acknowledge the support from the Russian Science Foundation (project No. 19-41-02005).

\end{document}